\documentclass[11pt]{article}

\usepackage{my_preamble}
\usepackage{amssymb, amsmath, amsfonts}
\usepackage{fullpage}

\newtheorem*{mainthm}{Main Theorem}

\newcommand{\pto}{\rightharpoonup}

\newcommand{\E}{\mathbf{E}}
\newcommand{\G}{\mathcal{G}}

\renewcommand{\C}{\mathcal{C}}

\renewcommand{\rho}{\varrho}

\newcommand{\req}{\mathtt{Req}}
\newcommand{\eval}{\mathtt{eval}}
\newcommand{\hdesc}{\mathtt{hdesc}}
\newcommand{\desc}{\mathtt{desc}}
\newcommand{\loc}{\mathtt{loc}}

\newcommand{\nlocsep}{\,\hbox{${\downarrow}\kern-.5em\hbox{\char'57}$}\,}
\newcommand{\namalg}{\,\hbox{${\amalg}\kern-.7em\hbox{\char'57}$}\,}

\begin{document}

\title{The Geometry of $L^k$-Canonization I: \\Rosiness from Efficient Constructibility}
\author{Cameron Donnay Hill\footnote{Correspondence to:   University of Notre Dame,
Department of Mathematics,
255 Hurley, Notre Dame, IN 46556. Email: \texttt{cameron.hill.136@nd.edu}. Telephone:  +1-574-631-7776}\\
Department  of  Mathematics,  University  of  Notre  Dame}
\date{ }
\maketitle

\begin{abstract}
We demonstrate that for the $k$-variable theory $T$ of a finite structure (satisfying certain amalgamation conditions), if finite models of $T$ can be recovered from diagrams of finite {\em subsets} of model of $T$ in a certain ``efficient'' way, then $T$ is rosy -- in fact, a certain natural $\aleph_0$-categorical completion $T^{\lim}$ of $T$ is super-rosy of finite $U^\thorn$-rank. In an appendix, we also show that any $k$-variable theory $T$ of a finite structure for which the Strong $L^k$-Canonization Problem is efficient soluble has the necessary amalgamation properties up to taking an appropriate reduct. \end{abstract}

\section*{Introduction}

This  article is the second of a three-part series (with \cite{me-2011a} and \cite{me-2012-canonization}) examining the model-theoretic geometry of an algorithmic problem -- the $L^k$-Canonization Problem. Here, $L^k$ denotes the fragment of first-order logic consisting of formulas with at most $k$ distinct variables, \emph{free or bound}, and it can be shown that for any finite structure $\M$ (in a finite relational signature $\rho$), its complete $k$-variable theory $Th^k(\M)$ is finitely axiomatizable in a uniform way. The $L^k$-Canonization Problem asks us to devise an operator $F$ that takes the theories $Th^k(\M)$ to finite models $F(Th^k(\M))\models Th^k(\M)$ -- thus, defining a ``canonical'' model of each complete $k$-variable theory that does have finite models. Composing the canonization operator $F$ with the mapping $\M\mapsto Th^k(\M)$, the operator $F(Th^k(-))$ can be thought of as a solution to a natural relaxation of the Graph Isomorphism Problem, the status of which is a major open problem in complexity theory (see \cite{graph-isom-survey} for an old survey).

Although this problem is known to be unsolvable over the class of \emph{all} $L^k$-theories (in particular, over all $L^3$-theories of finite structures -- see\cite{Grohe97largefinite}), it has been shown that for the class of stable $L^k$-theories  and for the class of super-simple $L^k$-theories with trivial forking dependence (with additional amalgamation assumptions), the $L^k$-Canonization Problem is {\em recursively} solvable (see \cite{koponen-2001} and \cite{koponen-2006}, respectively). 
In both of those cases, resolution of the $L^k$-Canonization Problem is reduced to showing that certain complete first-order theories associated with the original $L^k$-theories have the finite sub-model property. Thus, after the heavy lifting done by the model theory, the algorithm itself is extremely simple-minded. Moreover, the analyses in \cite{koponen-2001} and \cite{koponen-2006} do not assume \emph{a priori} that the $L^k$-theories in question certainly have finite models. In contrast, in this series of articles, we will examine the $L^k$-Canonization Problem for $L^k$-theories that do certainly have finite models. Moreover, we will consider implementation of $L^k$-Canonization operators in a significantly restricted model of computation, leading to a notion we call ``efficient constructibility.'' Finally, we will expand the original $L^k$-Canonization Problem to take whole $L^k$-elementary diagrams as input, which allows us to work with individual $L^k$-theories in a non-trivial way. Thus, the goal of this series of articles is to prove the following:
\begin{mainthm}
Let $\M_0$ be a finite structure, and let $K$ be the class of all finite models of $T = Th^k(\M_0)$. Assuming that $K$ has adequate amalgamation properties, let $T^{\lim}$ be the complete first-order theory of the direct limit of $K$. Then the following are equivalent:
\begin{enumerate}
\item $T^{\lim}$ is super-rosy of finite $U^\thorn$-rank.
\item $K$ is rosy.
\item The $K$-Construction Problem is solvable in a certain ``relational'' model of computation (based on \cite{abiteboul-vianu-1995} and exposed in \cite{me-2011b}):
\begin{quote}
\textsc{Given} $\M[A]$ for some (implicit) $\M\in K$ and $A\subseteq M$,\footnote{Here $\M[A]$ denotes the induced substructure of $\M$ on the subset $A$, so $\M[A]$ \emph{does not} carry any further information about $\M$.}

\textsc{Return} $\N\in K$ such that $A\subseteq N$ and $\N[A] = \M[A]$. 
\end{quote}
\end{enumerate}
\end{mainthm}

In \cite{me-2011a}, we established the content of parts 1 and 2 of the Main Theorem, and showed that, in fact, 1 and 2 are equivalent. In this article, we will prove the implication 
$$\textnormal{``Efficient constructibility''\,\,$\implies$\,\,1,2.}$$
We save for another day (i.e. \cite{me-2012-canonization}) the demonstrations that (1) a solution in the relational model of computation induces efficient constructibility in the sense of this article (immediate from the definitions), and (2) that the $K$-Construction Problem is in fact solvable by a relational Turing machine whenever $K$ is rosy (so, super-rosy of finite $U^\thorn$-rank

\begin{thm}
Let $\M_0$ be a finite structure, and let $K$ be the class of all finite models of $T = Th^k(\M_0)$. Assuming that $K$ has adequate amalgamation properties, let $T^{\lim}$ be the complete first-order theory of the direct limit of $K$. If $K$ is efficiently constructible, then $K$ is rosy (super-rosy of finite $U^\thorn$-rank).
\end{thm}

\section{Background and the Main Setting}

The contents of subsections \ref{subsec-start} and \ref{subsec-end} is taken \emph{verbatim} from the companion article, \cite{me-2011a}. Naturally, we require the settings to be identical. In the last subsection, we recall the necessary facts about \th-independence and rosiness that were established in \cite{me-2011a}.

\subsection{Finite-variable Logics}\label{subsec-start}

Finite-variable fragments of first-order logic, $L^k$, were formulated by many authors independently (e.g. \cite{poizat-1982}, but our main references have been \cite{otto-1997} and \cite{libkin-FMTbook-2004}). The importance of $L^k$ and its infinitary extension $L^k_{\infty,\omega}$ in finite-model theory is difficult to overstate. For our purposes, $L^k$ is satisfying because a ``complete'' $L^k$-theory -- that is, complete for $L^k$-sentences -- can have many non-isomorphic finite models, which is surely a prerequisite for bringing classical model-theoretic ideas to bear in finite-model theory.

\begin{defn}
Let $\rho$ be a finite relational signature. Assume $k\geq\ari(\rho)= \max \left\{\ari(R):R\in\rho\right\}$ and $k\geq 2$.
\begin{enumerate}
\item Fix a set $X = \{x_1,...,x_k\}$ of exactly $k$ distinct variables. Then, $L^X = L^X_\rho$ is the fragment of the first-order logic $L=L_\rho$ keeping only those formulas all of whose variables, \emph{free or bound}, come from $X$. If $V = \{x_0,x_1,...,x_n,...\}$ is the infinite set of first-order variables understood in the construction of the full first-order logic, then $L^k =\bigcup\left\{L^X:X\in{V\choose k}\right\}$, where ${V\choose k}$ is the set of $k$-element subsets of $V$.

As usual, we write $\phi(x_1,...,x_k)$ to mean that the set of free variables of $\phi$ is a subset of $\{x_1,...,x_k\}$, but not necessarily identical to it.

\item For a $\rho$-structure $\M$, the $k$-variable theory of $\M$, denoted $Th^k(\M)$ is the set of sentences $\phi$ of $L^k$ such that $\M\models\phi$. Note that $Th^k(\M)$ is complete with respect to $k$-variable sentences in that either $\phi\in Th^k(\M)$ or $\neg\phi\in Th^k(\M)$ for every $k$-variable sentence $\phi$.

\item For a $k$-tuple $\aa\in M^k$, we set
$tp^k(\aa;\M) = \left\{\phi(x_1,...,x_k)\in L^k:\M\models\phi(\aa)\right\}$
and if $T = Th^k(\M)$, then $S^k_k(T) = \left\{tp^k(\aa;\M):\aa\in M^k\right\}$. 

It can be shown -- in a number of ways -- that for a complete $L^k$-theory $T$, $T$ has a finite model only if $S^k_k(T)$ is finite. All of those methods also show that $S^k_k(T)$ is an invariant of $T = Th^k(\M)$ rather than $\M$ itself -- that is, if $Th^k(\N) = T$ for some other $\rho$-structure $\N$ (equivalently, if $\N\equiv^k\M$), then 
$\left\{tp^k(\bb;\N):\bb\in N^k\right\}= S^k_k(T)$, too. Finally, it can also be shown that if $\M$ is finite, then $Th^k(\M)$ is finitely axiomatizable, and in fact, the mapping $\M\mapsto Th^k(\M)$ is computable in $\textrm{Rel-}$\textsc{Ptime} (see \cite{abiteboul-vianu-1995}). This latter fact is the basis for our notion of efficient constructibility.

\item Let $\M$ be a $\varrho$-structure, and let $B\subseteq M$. Then for $X\in{V\choose k}$ as above and $e:X\to B\cup X$, we define $L^X(e)$ to be the set,
$$\left\{\phi(e(x_1),...,e(x_n)): \phi(x_1,...,x_n)\in L^X\right\}.$$
Then, 
$$L^k(B) = \bigcup_X\bigcup_{e:X\to B\cup X}L^X(e),$$
where, again, $X$ ranges over ${V\choose k}$. In particular, $L^k(B)$ \emph{is not} obtained by adding $B$ as collection of constant symbols to the underlying signature. 
Finally, we define
$$\diag^k(B;\M) =\left\{\phi\in L^k(B):\M\models\phi,\textrm{ $\phi$ has no free variables}\right\}.$$
We note that the object $\diag^k(B;\M)$ \emph{does not carry} the whole of $\M$ with it; this observation is crucial in understanding why the Strong $L^k$-Canonization Problem (below) is non-trivial.

\end{enumerate}
\end{defn}
Having defined the $k$-variable logic, we  define  the {\em $L^k$-Canonization Problem} by specifying what would amount to its solution. Thus, an {\em $L^k$-Canonization operator} (in the signature $\rho$) is a polynomial-time computable mapping $F:\fin[\rho]\to\fin[\rho]$ such that,
\begin{enumerate}
\item $F(\M)\equiv^k\M$ for all $\M\in\fin[\rho]$.
\item If $\M\equiv^k\N$, then $F(\M) = F(\N)$.
\end{enumerate}
(Here, $\fin[\rho]$ denotes the class of all finite $\rho$-structures.) The map $F$ selects a representative of each $\equiv^k$-class,  a ``canonical'' finite model of each complete $k$-variable theory that does actually have finite models.\footnote{In particular, $F$ is supposed to be invariant with respect re-encodings of a given structure. In the model of computation that we will work with in these articles, which works directly with unencoded structures, this issue will actually disappear entirely.} For $k\geq 3$, this is entirely impossible \cite{Grohe97largefinite}, so the natural move is to ask for canonization operators $F:K\to\fin[\rho]$ for sub-classes $K\subsetneq\fin[\rho]$. As stated, the operator $F$ really acts on theories, not structures, and this seems to be a serious impediment to studying the question from the point of view of ``geometric'' model theory; we therefore move again to a related problem which yields a better ``grip'' on the model theory of each $L^k$-theory.
Here, we call this the Strong $L^k$-Canonization Problem:
\begin{quote}
\textsc{Given} $\diag^k(A;\M)$ for some (implicit) $\M\in\fin[\rho]$ and $A\subseteq M$,\\
\textsc{Return} a model $\M'\models Th^k(\M)$ such that $A\subseteq M$ and $\diag^k(A;\M') = \diag^k(A;\M)$.
\end{quote}
(As usual, $\diag^k(A;\M)$, the $L^k$-elementary diagram of $A$ in $\M$, actually contains $Th^k(\M)$, as the sentences of $Th^k(\M)$ are precisely the $0$-ary formulas satisfied in $\M$ by $0$-tuples from $A$.) Of course, Strong $L^k$-Canonization initially appears to be a significantly harder problem than the original $L^k$-Canonization Problem, but we note that previous positive results on the $L^k$-Canonization Problem for $k=2$ and over restricted classes of $L^k$-theories have always yielded solutions of the Strong $L^k$-Canonization Problem with no additional toil -- see \cite{koponen-2001} and \cite{otto-1997b}.
We derive another advantage from the move to Strong $L^k$-Canonization in that it can be sensibly considered around a single fixed $L^k$-theory without trivializing the problem\footnote{There is an asymmetry in fixating on a single $L^k$-theory $T$. If we were initially studying classes $K\subseteq\fin[\rho]$ over which (Strong) $L^k$-Canonization has a solution, there is an implicit expectation of uniformity in that we would have expected a single algorithm to serve as an $L^k$-Canonization operator over the entire class. By fixing a single theory $T$, we allow ourselves to exploit additional properties of it that might be hard or impossible to read-off from its a finite presentation.}.

\subsection{Main Setting}\label{subsec-end}

The context for the the Main Theorem (for this and the two companion articles) is the following. Let $\MM_0$ be a countable $\aleph_0$-categorical structure in a finite relational signature $\rho$. We  assume that $Th(\MM_0)$ has the finite sub-model property in the following strong sense:
\begin{quote}
For every sentence $\phi$ and every finite $A_0\subset\MM_0$, $\phi\in Th(\MM_0)$ if and only if there is an algebraically closed finite $A\subset\MM_0$ such that $A_0\subseteq A$ and $A\models \phi$.
\end{quote}
Under this assumption, the $k$-variable theory $T = Th^k(\MM_0)$ has finite models of all sizes. We will also make two additional universality assumptions about $\MM_0$ with respect to finite models of $T$.
\begin{enumerate}
\item[U1.] For $B\subseteq\MM_0$, define 
$$\cl_{+1}^k(B) = B\cup \left\{a\in \MM_0: \diag^k(B;\MM_0)\models \textrm{$\tp^k(a/B)$ is algebraic}\right\},$$
$$\cl_0^k(B) = B,\,\,\cl_{n+1}^k(B) = \cl_{+1}^k(\cl^k_n(B)),\,\,\cl^k(B) = \bigcup_n\cl^k_n(B).$$
We require that for all $B\subset_\fin\MM_0$, $\acl(B) = \cl^k(B)$ and, if $B =\acl(B)$, then $tp(B)\equiv tp^k(B)\cup Th(\MM_0)$.
\item[U2.] $T\models Th_\forall(\MM_0)$.

This assumption has two invaluable consequences:
\begin{enumerate}
\item For every sufficiently large finite model $\A\models T$, there is an $L^k$-elementary embedding $\A\to\MM_0$.
\item For large enough finite $\B\models T$, for any $A\subseteq B$ such that $A = \cl^k(A)$, and any $f_0:A\to\MM_0$ -- a partial $L^k$-elementary map -- there is an $L^k$-elementary embedding $f:\B\to\MM_0$ extending $f_0$.
\end{enumerate}
\end{enumerate}

\noindent From all of these assumptions and appealing to \cite{baldwin-lessmann-2002} (Lemma 19 of that article, combined with the assumption on algebraic closures above), we may associate with $T$ a direct limit -- a countable structure  $\MM$ with the following properties:
\begin{enumerate}
\item[C1.] $\MM$ also satisfies the universality conditions expressed above.
\item[C2.] For any finite tuple $\aa$ from $\MM$, $\tp(\aa) \equiv \tp^k(\aa)\cup Th(\MM)$.
\item[C3.] Let $K_0$ be a set of representatives of every isomorphism-type of finite $L^k$-elementary sub-models of $\MM_0$. Then for any non-principal ultrafilter $\Psi$ on $K_0$, $\MM\equiv\Pi K_0/\Psi$.
\end{enumerate}
Abusing the terminology of \cite{baldwin-lessmann-2002} slightly, we call the theory $T^{\lim} = Th(\MM)$ the {\em canonical completion} of $T$. 

\begin{defn}
Let $K$ denote the class of (up to isomorphism) finite $L^k$-elementary sub-models of $\MM$ quasi-ordered by the $L^k$-elementary substructure relation. That is to say, starting with the set of finite $L^k$-elementary sub-models of $\MM_0$, we close under isomorphisms and obtain an object $(K,\preceq^k)$ that is much like an abstract elementary class (AEC). By definition, the model theory of the $k$-variable theory $T$ is, effectively, just the model theory of the class $(K,\preceq^k).$\footnote{In fact, if $T$ is any $k$-variable theory whose \emph{finite models} have the Joint-Embedding and Amalgamation Properties indicated by the two consequences of U2, it is still possible to generate a limit model $\MM$ as above; given $\MM$, one can then go through this entire process with $\MM_0=\MM$. Thus, we can also view the model theory of $Th(\MM)$ as a by-product of the model theory of $T$ or of $(K,\preceq^k)$. This is the approach taken in \cite{baldwin-lessmann-2002}.}
For $\M\in K$ and $C\subseteq M$, we define,
$$K_{(C;\M)} = \left\{\N\in K: C\subseteq N,\,\diag^k(C;\N) = \diag^k(C;\M)\right\}.$$
Writing $K_{(C;\M)}$ over and over again is rather cumbersome, so we often write $K_C$ -- taking $\M$ to be clear from context or assuming that a $k$-variable diagram for $C$ is fixed. Also, we will often identify the pair $(C;\M)$ both with the set $C$ and the diagram $\diag^k(C;\M)$.
\end{defn}

\begin{obs}\label{obs:almost-EI}
There is a finite set $\E$ of 0-definable equivalence relations such that for any other 0-definable equivalence relation $E(\yy_1,\yy_2)$, there is a boolean combination $\phi(\yy_1,\yy_2)$ of members of $\E$ such that $T^{\lim}\models E(\yy_1,\yy_2)\bic \phi(\yy_1,\yy_2)$. In particular, let $\MM_\E$ be the reduct of $\MM^\eq$ consisting of the home-sort and the finitely many imaginary sorts corresponding to equivalence relations in $\E$; then $Th(\MM_\E)$ eliminates imaginaries.

Moreover, each equivalence relation $E$ in $\E$ is expressible as boolean combination of formulas of $L^k$. Thus, for $\A\preceq^k\MM$ and $\aa,\bb\in A$, $\A\models E(\aa,\bb)$ if and only if $\MM\models E(\aa,\bb)$; this is true even when $\A$ is a finite substructure. Thus, each finite $L^k$-elementary substructure $\A$ of $\MM$, naturally expands to a substructure $\A_\E$ of $\MM_\E$ carrying all of the relevant information from $\MM^\eq$. In this way, we also obtain an expansion of the class $(K,\preceq^k)$ (see below) to $K_\E$, and where the ambiguity is not dangerous, we will identify $K$ and $K_\E$.
\end{obs}

In model theory, it is a common practice it is a common practice to Morley-ize a theory under study -- adding relation symbols for every 0-definable relation so that the theory automatically has quantifier-elimination. Indeed, subsequently taking the reduct to the language consisting of only the new relation symbols still does not change anything about the theory essentially. Here, we implicitly take a similar approach, working in a language in which the $k$-variable $k$-types are named as relation symbols as are the projection maps to the finitely-many needed imaginary sorts. As a point of interest, while this move is already justified in our Main Setting, it is possible to show that a ``capped'' $L^k$-theory -- in which every diagram of a finite subset of a model extends to a finite model -- \emph{automatically} falls into our Main Setting. This is discussed in Appendix \ref{app1} below.

%
%
%
%

\subsection{Rosiness, model-theoretic independence relations and geometric elimination of imaginaries}

To conclude this section, we recall some definitions and facts from \cite{me-2011a} about super/rosiness of $L^k$-theories and their limit theories consistent with Main Setting. We recall, in particular, we essentially defined the phrase, ``$K$ is rosy,'' for $K$ the class of finite models of $T$,  to mean, ``$T^{\lim}$ is rosy'' (this is in contrast to the developments in \cite{hyttinen-2000} and \cite{koponen-2006}). For the sake of brevity, then, we make several definitions in the context of the limit theory and its models, but we don't mention their translations into the context of the underlying class $K$ itself.

\begin{defn}
[Model-theoretic independence relation]
Let $T$, $K$ and $T^{\lim}$ be as in the Main Setting, and let $\MM\,(=\MM^\eq)$ be a countable model of $T^{\lim}$. An independence relation $\oind$ is a subset of $\left\{ (A,B,C):A,B,C\subset_\fin \MM\right\}$ satisfying the following for all $A,B,C\subset_\fin \MM$:
\begin{enumerate}
\item[] Invariance: If  $\sigma\in Aut(\MM)$, then $A\oind_CB$ iff $\sigma[A]\oind_{\sigma[C]}\sigma[B]$.
\item[] Preservation of algebraic dependence: If $A\oind_CB$, then $\acl(AC)\cap \acl(BC)=\acl(C)$.
\item[] Existence: $A\oind_CC$.
\item[] Extension: If $A\oind_CB$ and $BC\subseteq D\subset_\fin M$, then there is an $A'\equiv_{BC}A$ such that $A'\oind_CD$.
\item[] Symmetry: $A\oind_CB$ if and only if $B\oind_CA$
\item[] Transitvity: If $C_0\subseteq C\subseteq B$, then $A\oind_{C_0}B$ iff $A\oind_{C_0}C\,\&\,A\oind_CB$.
\item[] Local character: For every finite family $S = \{S_0,...,S_{n-1}\}$ of sorts of $\M^\eq$, there is a function $g_S:\omega\to\omega$ such that if $A,B\subseteq \bigcup_iS_i^{\M^\eq}$, then there is a subset $B_0\subseteq B$ such that $|B_0|\leq g_S(|A|)$ and $A\oind_{B_0}B$.
\end{enumerate}
The relation $\oind$ is a a notion of independence if it satisfies Invariance, Preservation of algebraic dependence, Extension, and,
\begin{enumerate}
\item[] Partial right-transtivity: If $C_0\subseteq C\subseteq B$, then $A\oind_{C_0}B$ implies $A\oind_{C_0}C\,\&\,A\oind_CB$.
\end{enumerate}
but not necessarily Symmetry, (full) Transitivity or Local character.
\end{defn}
In \cite{me-2011a}, we showed that for $T$, $T^{\lim}$ as in the Main Setting the restriction to triples of finite sets in the definition of an independence relation is fully consistent with the usual definitions for independence relations. Further, we found that in this setting, there is no actual need to prove Local Character -- it follows from the other axioms in this strong (``super'') way. In the next sections, we will build, from the assumption of ``efficient construbility,'' a notion of independence $\dind$ which is symmetric and transitive. From this, it will follow that $T$ is rosy. The key theorem from \cite{me-2011a} that we need here is the following; it is built upon on similar theorems of \cite{ealy-onshuus-2007} and \cite{adler-2007}:

\begin{thm}[Rosy = ``has an independence relation'']
Let $T$, $T^{\lim}$, and $K$ be as in the Main Setting. The following are equivalent:
\begin{enumerate}
\item $K$ is rosy.
\item $K$ admits an independence relation (without reference to Local Character).
\item $T^{\lim}$ is rosy.
\item $T^{\lim}$ admits an independence relation.
\item $T^{\lim}$ admits a notion of independence with Symmetry and Transitivity.
\item $T^{\lim}$ admits a notion of independence with Local character.
\end{enumerate}
(In either of the last two cases, that notion of independence is, then, an independence relation.)
\end{thm}

\section{Definition of `Efficient constructibility'}\label{sec:def-of-EC}

As  noted in the introduction, our remaining goal for this article is to show that ``efficiently constructible'' classes $K$ (arising as in the Main Setting along with $T$, $T^{\lim}$) are rosy by constructing a notion of independence with Symmetry and Transitivity. For this, of course, we will need to make several definitions leading up to the definition of efficient constructibility. 

We should also apologize to the reader for our use of the word ``efficient'' here. It is well-known to finite-model theorists that the logic FO+IFP -- first-order logic extended with the inflationary fixed-point operator -- captures \emph{relational}-\texttt{PTime} (see \cite{abiteboul-vianu-1995}); thus, the central role of the IFP-operator in our definition already suggests that the word ``efficient'' is more-or-less appropriate in that sense. 

From now on, we fix $T$, $K$ and $T^{\lim}$ as in the Main Setting with a fixed finite relational signature $\rho$, and $\MM$ always denotes a countable model of $T^{\lim}$. We define 
$$\acl[\MM] = \left\{\acl(A): A\subset_\fin \MM\right\}$$
which is a family of finite sets. In this presentation, we will not make any real use of Morley-ization of the langauge in the sense of Appendix \ref{app1}.

%
%

\begin{defn}[Inflationary fixed-points of expanded formulas]
Let $X^{(r)}$ be a new relation symbol, not in $\varrho$. Then a formula, say $\psi(\xx) = \psi(x_0,...,x_{r-1};X)$, of $L[\varrho\cup\{X\}]$ will be called a proper expanded formula. For $A\in\acl[\MM]$, we define a subset $\psi^\infty[A]$ of $A^r$ as follows:
\begin{itemize}
\item $\psi^0[A] = \emptyset$
\item Given $\psi^t[A]\subseteq A^r$, define
$$\psi^{t+1}[A] = \psi^t[A]\cup\left\{\aa\in A^r:(A,\psi^t[A])\models \psi(\aa)\right\}$$
where $(A,\psi^t[A])$ is the $\varrho\cup\{X\}$-expansion of $A$ in which $X$ is interpreted is $\psi^t[A]$.
\item $\psi^\infty[A] = \bigcup_{t<\omega}\psi^t[A] = \psi^{|A|^{O(1)}}[A]$
\end{itemize}
Note that that for every $t<\omega$, the set $\psi^t[A]$ and $\psi^{t+1}[A]\setminus\psi^t[A]$ are first-order $A$-definable sets. We also note that satisfaction, here, is evaluated with respect to the induced substructure $A$, rather than with respect to $\MM$.
\end{defn}

\begin{defn}
A {\em pre-program for $K$} will consist of the  data,
$$\Pi = (\varrho^{test},\varrho^+,\Theta = \{\theta_i\}_{i<N_0},\Phi = \{\phi_i\}_{i<N},\Sigma = \{\psi_\sigma,\xi_\sigma\}_{\sigma:\Phi\to2},\yy = (y_0,...,y_{r-1}))$$
as follows:
\begin{itemize}
\item $\varrho^{test} = \varrho\cup\{X_0^{(r)},...,X_{N_0-1}^{(r)}\}$ is an expansion of $\varrho$ and for each $i<N$,  $\theta_i(x_0,...,x_{r-1};X_i^{(r)})$  is a proper expanded  $\exists$-formula of $L[\varrho\cup\{X_i\}]$.
\item  $\Phi = \left\{\phi_0,...,\phi_{N-1}\right\}$ is a family of sentences of $L[\varrho^{test}]$.
\item  $\varrho^+ = \varrho\cup\{Y^{(r)}\}$ is another expansion of $\varrho$; for each map $\sigma:\Phi\to2$,  $\psi_\sigma(x_0,...,x_{r-1})$ is a proper expanded $\exists$-formula of $L[\varrho^+]$;  $\xi_\sigma(\yy,z_0,...z_{m-1})$ is an arbitrary formula of $L[\varrho^+]$.
\end{itemize}
For $A\subset_\fin \MM$, we write $A^{test}$ for the expansion $(A,\theta_0^\infty[A],...,\theta_{N_0-1}^\infty[A])$, and we write $A^\sigma$ for the expansion $(A,\psi_\sigma^\infty[A])$.

Suppose $A = \acl(A)\subset_\fin \MM$. Write $\sigma_A$ for the function $\Phi\to2$ given by 
$$\sigma_A(\phi_i) = 1\iff A^{test}\models \phi_i$$
for each $i<N$. We say that $A$ {\em requests attention at $(a_0,...,a_{r-1})\in A^r$} just in case,
$$\aa\in \psi_{\sigma_A}^\infty[A]\textrm{ and }A^{\sigma_A}\models \neg\exists \yy\,\xi_{\sigma_A}(\yy,a_{0},...,a_{r-1}).$$
If $A$ does not request attention at any $\aa\in A^r$, then we say that $A$ is $\Pi$-complete. Finally, $\Pi$ is genuinely a pre-program for $\K$ if  for any $A\in\acl[\MM]$
$$\textrm{$A$ is $\Pi$-complete iff $A\models T$ (iff $A\in K$)}.$$
\end{defn}
%
%


\begin{defn}
Let $\Pi = (\varrho^{test},\varrho^+,\Theta,\Phi,\Sigma,\yy)$ be a pre-program for $K$, and let $\M\models T$. A command operator for $\Pi$ is a map,
$$F:\acl[\M]\longrightarrow \P_\fin(S_E^\M)$$
satisfying the following whenever $A = \acl^\M(A)\subset_\fin M$.
\begin{itemize}
\item If $A$ is $\Pi$-complete, then $F(A) = \emptyset$.
\item Suppose $A$ is $\Pi$-incomplete. Up to an \emph{a priori} fixed linear order of quantifier-free $r$-types of $\varrho^{test}$ over $\emptyset$, suppose $A$ requests attention at $\aa\in A^r$, and $\pi = qtp(\aa;A^{test})$ is minimal for this condition. (In what follows, we will say that $(\Pi,F)$ acts on $A$ at $\aa$.)
Then $F(A)$ is of the form,
$$\left\{(\bb_{\cc},\cc)/E:\cc\in \req(\sigma_A,\aa)\right\}$$
where $E =E_{\sigma_A,\pi}$ is a 0-definable equivalence relation on $r$-tuples and
$$\req(\sigma_A,\aa) = \left\{\cc\in A^r:A^{test}\models \pi(\cc)\,\&\,A^{\sigma_A}\models \neg\exists \yy\,\xi_{\sigma_A}(\yy,\cc)\right\}$$
and $\bb_{\cc}\in M^r$ for each $\cc\in\req(\sigma_A,\aa)$.
  Finally, we require that given $\sigma_A$, the pairs $\left(\cc,\,(\bb_{\cc},\cc)/E\right)$ are elements of an $\aa'$-definable relation (though not necessarily a function) for some $\aa'\models\pi$ in $A$.
\end{itemize}
 Together, $(\Pi,F)$ is a {\em program} for $K$, which then defines an operator 
$$\eval = \eval_{\Pi,F}:\acl[\MM]\longrightarrow\acl[\MM]$$
as follows: For $A \in\acl[\MM]$,
$$\eval(A) = \acl\left(A\cup F(A)\right)\cap \MM$$

For $A=A_{-1}\subset_\fin M$, $(\Pi,F)$ induces a sequence $\eval^*(A)=(A_i)_{i<\omega}$ as follows:
$$A_0=\acl^\M(A_{-1}),\,\,A_{i+1} =\eval(A_i)\textrm{ for all $i<\omega$}.$$
\end{defn}

\begin{defn}
We say that $K$ is {\em efficiently constructible} if there is a program $(\Pi,F)$ for $K$ such that the following conditions are satisfied:
\begin{enumerate}
\item There is a polynomial $p(t)\in\NN[t]$ such that for any $A\subset_\fin \MM$, if $(A_i)_{i<\omega} = \eval^*(A)$, then $A_i = A_{p(|A|)}$ whenever $p(|A|)\leq i<\omega$.

\noindent (We may then define $A_{\Pi,F}:= A_{p(|A|)}$, and we note that $A\leq A_{\Pi,F}\leq_\fin \M$ and $A_{\Pi,F}\in K$.)

\item $(\Pi,F)$ is weakly automorphism invariant: For $A,B\subset_\fin \MM$, if there is an elementary bijection $f:A\to B$, then $A_{\Pi,F}\cong B_{\Pi,F}$ via an isomorphism extending $f$.

\item For all $A\subseteq B\subset_\fin M$, there is an automorphism $g\in Aut(\M/A)$ such that $g[A_{\Pi,F}]\subseteq B_{\Pi,F}$.
\end{enumerate}
\end{defn}

%
%

\subsection{Induction graphs and construction graphs}

Much of the proof of theorem \ref{tim:ec-implies-rosy} consists in defining a certain family of directed acyclic graphs (DAGs) -- construction graphs -- that represent all of the relevant model-theoretic information in runs of a program $(\Pi,F)$. (See  Appendix \ref{app:dags} for background on DAGs and $d$-separation.) Eventually, we will recover an independence relation $K$ (hence for $T^{\lim}$) by synthesizing two previously existing ideas. The first of these is that of $d$-separation in DAGs, which is well-known in statistical learning theory under the rubric of ``graphical models.'' In a fixed finite DAG, $d$-separation has several properties in common with model-theoretic independence relations (although Extension and Invariance don't make any sense there). The second idea is the definition of \th-indepedence, $\thind$, found in \cite{adler-2007}:
\begin{itemize}
\item $A\indep^\textsc{m}_CB$ iff $\acl(AC')\cap\acl(BC') = \acl(C')$ whenever $C\subseteq C'\subseteq \acl(BC)$
\item $A\thind_CB$ iff for every $D\supseteq BC$, there is an $A'\equiv_{BC}A$ such that $A'\indep^\textsc{m}_CD$.
\end{itemize}
In effect, we will just replace $\indep^\textsc{m}$ by a perturbation of $d$-separation in our construction graphs and play the same game as in \cite{adler-2007}. A certain easy observation about the definability of our construction graphs (in $\MM=\MM^\eq$) will guarantee that we can win this game. 

Finally, we must remark that our construction seems to be only loosely tied to the original $K$-Construction problem (or Strong $L^k$-Canonization). Accordingly, it may be of some more general interest to investigate a larger world of connections between rosiness in, say, $\aleph_0$-categorical theories and constructibility conditions defined using inflationary fixed-points of formulas.

\begin{defn}
Let $\phi(x_0,...,x_{r-1};X^{(r)})$ be a formula of $L[\varrho\cup\{X\}]$. Let $A\subset_\fin \MM$, $R\subseteq A^r$ and 
$$R' = R\cup \left\{\aa\in A^r:(A,R)\models \phi(\aa)\right\}.$$
For $\aa,\bb\in A^r$, we define
$$\textrm{``$\bb\forces^\phi\aa\in R'$ w.r.t. $(A,R)$''}$$
to mean that:
\begin{enumerate}
\item $\aa\in R'\setminus R$ and $\bb\in R$;
\item $(A,R)\models \phi(\aa)$;
\item $(A,R_{(\bb)})\models \neg\phi(\aa)$ where $R_{(\bb)} = R\setminus \acl(\bb)^r$.
\end{enumerate}
Assuming $\phi^\infty[A] = \phi^e[A]$, we define $IG_\phi[A,e]$ to be the directed graph with vertex set $e{+}1\times A^r$ and edge set $\bigcup_{t<e}E^t_\phi[A,e]$, where for each $t<e$, 
$$E^t_\phi[A,e] = \bigcup\begin{cases}
\left\{\,\,\left((t,\aa),(t{+}1,\aa)\right):\aa\in A^r\right\}\\
\left\{\,\,\left((t,\bb),(t{+}1,\aa)\right):\textrm{$\bb\forces^\phi\aa\in \phi^{t+1}[A]$ w.r.t. $(A,\phi^t[A])$}\right\}\\
\end{cases}$$
\end{defn}

\begin{defn}
To define the induction graph of $\Pi$ on $A$, denoted $IG_{\Pi}[A]$, we first set 
$$e^\Pi = \min\left\{e:\bigwedge_{i<N_0}\theta_i^\infty[A] =\theta_i^e[A]\right\}.$$
Then $IG_\Pi^0[A]$ has vertex set $e^\Pi{+}1\times A^r$ and edge set
$$E_\Pi[A] = \bigcup_{i<N_0}E_{\theta_i}[A,e]$$
Now, let $\sigma = \sigma_A$ and $e_A =\min\left\{e:\psi_\sigma^\infty[A] =\psi_\sigma^e[A]\right\}$. Then $IG_\Pi[A]$, with vertex set $(e^\Pi{+}e_A{+}1)\times A^r$, is obtained by gluing $IG_{\psi_\sigma}[A,e_A]$ to the end of $IG^0_\Pi[A]$ in the natural way.
\end{defn}

\begin{defn} 
Let $A\subset_\fin \MM$.
We first define $CG^0[A] = CG^0_{\Pi,F}[A]$ as follows:
\begin{enumerate}
\item Let $\eval^*(A) = (A_i)_{i<\omega}$ and set $p = p(|A|)$. (For the statement here, we assume that $A_i$ is $\Pi$-incomplete for each $i\geq p(|A|)$.)
\item  For each $i=0,1,...,p-1$, we define an enlargement $IG_\Pi[A_i]^+$ of the induction graph $IG_\Pi[A_i]$. Assuming  $(\Pi,F)$ acts on $A_i$ at $\aa_i\in A_i^r$ and setting $\pi_i = qtp(\aa_i;A_i^{test})$, we define
$$C_i = \left\{\cc_i\in A_i^r: A_i^{test}\models \pi_i(\cc)\,\&\,A_i^{\sigma_{A_i}}\models \neg\exists \yy\,\xi_{\sigma_{A_i}}(\yy,\cc) \right\}$$
and let $\cc\mapsto (\dd_{\cc},\cc)/E$ be the map associated with the application of $F$.
To make $E_\Pi[A_i,\aa_i]^+$, we just add a new layer to $E_\Pi[A_i]$: the set of new vertices is  $\{\star_i\}\times A_{i+1}^r$, and the new edges are of the form
$$\left(\,(e^\Pi(A_i){+}e_{A_i}{+}1,\cc), (\star_i,\bb)\,\right)$$
for $\cc\in C_i$ and $\bb\in A_{i+1}^r\cap \acl^{\M^\eq}\left( \,(\dd_{\cc},\cc)/E\,\right)$.

\item We obtain $CG^0[A]$ be gluing/identifying each vertex $(\star_i,\aa)$ of $E_\Pi[A_i]^+$ with the vertex $(0,\aa)$ of $E_\Pi[A_{i+1},\aa_{i+1}]^+$.
\end{enumerate}
Next, we define the construction graph $CG[A]$ by pruning $CG^0[A]$ as follows:
\begin{enumerate}
\item $CG[A]$ retains only the following vertices:
\begin{itemize}
\item The vertices of the very first layer (those from the first layer of $E_\Pi[A_0]^+$), which we now write as $(\star_{-1},\aa)$.
\item Those of the form $(\star_i,\bb)$ for some $i\leq p$
\end{itemize}
\item For $i=-1,0,...,p-1$, $\left(\,(\star_i,\aa),(\star_{i+1},\bb)\,\right)$ is an edge of $CG[A]$ if one of the following holds:
\begin{itemize}
\item $rng(\aa) = rng(\bb)$;
\item There is a (directed) path from $(\star_i,\aa)$ to $(\star_{i+1},\bb)$ in $E_\Pi[A_i,\aa_i]^+$, \emph{\textbf{and}}
 $\aa\in\acl(\bb)$.
\end{itemize}
\end{enumerate}
\end{defn}

We will need to do a bit more bookkeeping in our construction graphs than we can easily express without some further machinery. Thus, we define a certain ``tuple-labeling function'' that connects everything produced in a run of the program $(\Pi,F)$ on $A\subset_\fin M$ back to the original set $A$. From this, we define the notion of ``hereditary descendants'' in our construction graphs, which will be indispensable in both the definition and the analysis of our independence relation.

\begin{defn}
Tuple-labeling maps $\lambda^A:V(CG[A])\to H(A)$ and hereditary descendents.

For a non-empty set $X$, we define $H(X)$ as follows:
\begin{itemize}
\item $H_0(X) = (\{0\}\times X)^{<\omega}$
\item $H_{n+1}(X) = \left\{\left(\,(n{+}1,\aa_0),...,(n{+}1,\aa_{k-1})\,\right):\aa_0,...,\aa_{k-1}\in H_n(X), 0<k<\omega\right\}$
\item $H(X) = \bigcup_{n<\omega}H_n(X)$.
\end{itemize}
We define $\lambda^A:V(CG^0[A])\to H(A)$ as follows.
\begin{itemize}
\item Let $\lambda_0^A$ be just $(\star_{-1},\aa)\mapsto \left((0,a_0),...,(0,a_{r-1})\right)$
\item Given $\lambda_i^A$: for a vertex $(\star_{i+1},\bb)$, let $(\star_i,\aa_0),...,(\star_i,\aa_{n-1})$ enumerate $(\star_{i+1},\bb)$'s predecessors in $CG[A]$. Then set
$$\lambda^A_{i+1}(\star_{i+1},\bb) = \left(\,(i{+}1,\lambda^A_i(\aa_0)),...,(i{+}1,\lambda^A_i(\aa_{n-1}))\,\right).$$
\item Set $\lambda^A = \bigcup_{i\leq p}\lambda^A_i$
\end{itemize}
Finally, if $A_0\subseteq A$, we define the set of {\em hereditary descendants} of $A_0$ with respect to $A$, $\hdesc(A_0;A)$ inductively as follows:
\begin{enumerate}
\item $\{\star_{-1}\}\times A_0^r\subseteq\hdesc(A_0;A)$
\item If $\lambda^A(\star_{i+1},\bb) = \left(\,(i{+}1,\lambda^A_i(\star_i,\aa_0)),...,(i{+}1,\lambda^A_i(\star_i,\aa_{n-1}))\,\right)$ and $(\star_i,\aa_j)\in \hdesc(A_0;A)$ for some $j<n$, then $(\star_{i+1},\bb)\in \hdesc(A_0;A)$
\end{enumerate}
\end{defn}

The following observation is very important, even though it follows very easily from noticing that (i) finitely many iterations of the inflationary fixed-point operation amount to a first-order formula, and (ii) given $A\subset_\fin \MM$, each of the hypothetical scenarios $A\subseteq B\subset_\fin \MM$ can be encoded as a first-order statement.

\begin{obs}[Definability of construction graphs]\label{obs:def-of-CGs}
For each $k<\omega$, there is a finite family of $X^k_0,...,X^k_{n_k-1}$ of sorts of $\MM^\eq$ such that the following holds: Let $A\subseteq B\subset_\fin \MM$. Then for any $g\in Aut(\MM/A)$ such that $g[A_{\Pi,F}]\subseteq B_{\Pi,F}$, there is a mapping,
$$f^A: \hdesc(A;B)\longrightarrow\dcl^\eq(g[A_{\Pi,F}])$$
satisfying the following conditions:
\begin{enumerate}
\item If $k\neq \ell$, then $\{X^k_i\}_{i<n_k}\cap \{X^{\ell}_j\}_{j<n_{\ell}}=\emptyset$
\item For each $k\leq p(|B|)$ and $(\star_k,\aa)\in (\{\star_k\}{\times}B^r_k)\cap \hdesc(A;B)$, $f^A(\star_k,\aa)\in X^k_0\,{\dot\cup}\,\cdots\,{\dot\cup}\, X^k_{n_k-1}$
\item For each $k\leq p(|B|)$, and $(\star_k,\aa_1),(\star_k,\aa_2)\in \{\star_k\}{\times}B_k^r\cap \hdesc(A;B)$,
$$\tp(f^A(\star_k,\aa_2)/A) = \tp(f^A(\star_k,\aa_2)/A)\textnormal{ in $\MM = \MM^\eq$}$$
implies
$$\tp_{\exists}(\aa_1;B_k^{test})=\tp_{\exists}(\aa_2;B_k^{test})$$
where $\tp_\exists(-;B_k^{test})$ indicates a partial type consisting of existential formulas with satisfaction evaluated in the finite structure $B_k^{test}$.
%
\item For any $k\leq p(|B|)$ and  $(\star_k,\aa)\in (\{\star_k\}{\times}B_k^r)\cap \hdesc(A;B)$, 
$\acl(\aa)\subseteq\acl\left(A\cup\left\{f^A(k,\aa)\right\}\right)$
\end{enumerate}
\end{obs}

We derive the following lemma as an immediate consequence of Observation \ref{obs:def-of-CGs} and the fact that $T$ is $\aleph_0$-categorical, hence uniformly locally finite. We refer the reader to  Appendix \ref{app:dags} for the notion of $d$-separation and the notation $[-\amalg-|-]$.
\begin{lemma}\label{lemma:bdd-local-wt}
There is function $\loc:\omega\times\omega \to\omega$ such that for all $A,B\subseteq C\subset_\fin \MM$, 
 and any $\cc_0,...,\cc_n\in C^r$ with $n\geq \loc(|A|,|B|)$, if:
$$\left[(\star,\cc_i)\,\amalg\,(\star,\cc_j)\,\,|\,\,\hdesc(AB;C)\right]_{CG[C]}$$
for all $i<j\leq n$ (setting $\star = \star_{-1}$ forevermore), then 
$$\left[(\star,\cc_i)\,\amalg\,\{\star\}{\times} A^r\,\,|\,\,\hdesc(B;C)\right]_{CG[C]}$$
for some $i\leq n$.
\end{lemma}

\section{Deviation and $d$-Independence}\label{sec:dev-and-d-indep}

In this section, we use the construction graphs developed  above to define an independence relation $\dind$ on our efficiently constructible class $K$. To do this, we first make a perturbation $\downarrow$ -- called local separation, for lack of inspiration -- of the $d$-separation relation; thereafter, we use local separation as a substitute for the $\indep^\textsc{m}$ relation of \cite{adler-2007}, and we proceed along similar lines.
\begin{defn}
Let $A,B,C\subseteq D=\acl(D)\subset_\fin M$. We say that $A,B$ are locally separated over $C$ in $D$ -- denoted $A\downarrow_CB/D$ -- if 
$$\left[\,\{\star\}{\times}A^r\,\amalg\,\{\star\}{\times}B^r\,\,|\,\,\hdesc(\acl(C');D)\right]_{CG[D]}$$
whenever $C\subseteq C'\subseteq \acl(BC)$. 

Again with $B,C\subseteq D\subset_\fin M$, though $D$ need not be algebraically closed, let $\pi(\xx)$ be a partial type over $BC$. We define $\Delta_0[\pi(\xx),C]_D$ to be the set of finite sets $D\subseteq D'\in\acl[\M]$ such that:
\begin{enumerate}
\item $\pi(\xx)$ {\em is} realized  in $D'$;
\item For every realization $\aa$ of $\pi(\xx)$ in $D'$,
$$\aa\downarrow_{BC}D/D'\,\,\implies\,\, \aa\nlocsep_CD/D'$$
(using $\aa$ as shorthand for the set $rng(\aa)$).
\end{enumerate}
We define $\Delta[\pi,C]_D$ to the be the set of isomorphism types over $D$ appearing in $\Delta_0[\pi,C]_D$, and we call $\Delta[\pi,C]_D$ {\em the deviation of $\pi$ over $C$ with respect to $D$}.

Finally, we define $A\dind_CB$ to mean that $\Delta[p,C]_D$ is finite whenever $ABC\subseteq D\subset_\fin M$, $p = tp(\aa/BC)$ and $\aa$ is an enumeration of $A$. We also say, then, that $p$ does not deviate over $C$.
\end{defn}

\begin{lemma}
Let $A,B,C\subseteq D\subset_\fin M$ where $D = \acl(D)$. 
\begin{enumerate}
\item[] {\em $\downarrow$-Monotonicity}: If $A\downarrow_CB/D$ and $B_0\subseteq B$, then $A\downarrow_CB_0/D$.
\item[] {\em $\downarrow$-Base-monotonicity}: If $A\downarrow_CB/D$ and $B_0\subseteq B$, then $A\downarrow_{CB_0}B/D$.
\item[] {\em Preservation of algebraic dependence} If $A\downarrow_CB/D$, then $\acl(AC)\cap\acl(BC)=\acl(C)$.
\end{enumerate}
\end{lemma}
\begin{proof}[Proof of $\downarrow$-Monotonicity]
Assuming $A\downarrow_CB/D$, we consider $C\subseteq X\subseteq \acl(B_0C)$. Clearly, $C\subseteq X\subseteq \acl(BC)$, so 
$$\left[\,\{\star\}{\times}A^r\,\amalg\, \{\star\}{\times}B^r\,\,|\,\,\hdesc(\acl(X);D)\,\right]_{CG[D]}.$$
By the monotonicity of $[\,-\,\amalg\,-\,|\,-\,]$ (see the appendix), we find 
$$\left[\,\{\star\}{\times}A^r\,\amalg\, \{\star\}{\times}B_0^r\,\,|\,\,\hdesc(\acl(X);D)\,\right]_{CG[D]}$$
and this suffices.
\end{proof}

\begin{proof}[Proof of $\downarrow$-Base-monotonicity]
Assuming $A\downarrow_CB/D$, we now consider $B_0C\subseteq X\subseteq \acl(BC)$. Again, $C\subseteq X\subseteq\acl(BC)$, so 
$$\left[\,\{\star\}{\times}A^r\,\amalg\, \{\star\}{\times}B^r\,\,|\,\,\hdesc(\acl(X);D)\,\right]_{CG[D]}$$
follows from the assumption.
\end{proof}

\begin{proof}[Proof of preservation of algebraic dependence]
For the contrapositive, suppose $e\in \acl(AC)\cap\acl(BC)$ and $e\notin\acl(C)$. Then $\underset{\textrm{$n$ times}}{\underbrace{(e,...,e)}}$ witnesses, 
$$\left[\,\{\star\}{\times}A^r\,\namalg\, \{\star\}{\times}B^r\,\,|\,\,\hdesc(\acl(X);D)\,\right]_{CG[D]}$$
so $A\nlocsep_CB/D$.
\end{proof}

In preparation for some slightly more complex arguments, we collect a few useful observations
\begin{obs}\label{obs:dev-containment}
Let $B,C\subseteq D\subseteq D_1\subseteq E\in\acl[\M]$, and let $\pi_0(\xx)\subseteq \pi(\xx)$ be partial types over $BC$. 
\begin{enumerate}
\item If $E\in \Delta[\pi,C]_{D}$, then $E\in\Delta[\pi,C]_{D_1}$
\item If $E\in\Delta[\pi_0,C]_D$, then $E\in\Delta[\pi,C]_D$.
\end{enumerate}
\end{obs}


\begin{obs}\label{obs:alg-extensions}
Let $\pi(\xx)$ be a partial type over $C\subset_\fin M$. If every complete extension of $\pi$ to $C$ is algebraic, then $\pi$ is algebraic.
\end{obs}

The following proposition just collects together what we know of $\dind$ so far; that is, $\dind$ is a notion of independence -- if not an independence relation -- even if the underlying program $(\Pi,F)$ does not actually witness efficient constructibility.

\begin{prop}\label{prop:properties-without-bdd-loc-wt}
$\dind$ has the following properties:
\begin{enumerate}
\item[] {\em Invariance}: If $A,B,C\subset_\fin M$ and $\sigma\in Aut(\M)$, then 
$A\dind_CB\iff \sigma[A]\dind_{\sigma[C]}\sigma[B]$.

\item[] {\em Preservation of algebraic dependence}: For $A,B,C\subset_\fin M$, if $A\dind_CB$, then 
$$\acl(AC)\cap \acl(BC) = \acl(C).$$

\item[] {\em Existence}: If $A,C\subset_\fin M$, then $A\dind_CC$.

\item[] {\em Extension}: For $A,B,B_1,C\subset_\fin M$ with $BC\subseteq B_1$, if $A\dind_CB$, then there is an $A'\equiv_{BC}A$ such that $A'\dind_CB_1$.
\item[] {\em Partial right-transitivity}: For $A,B_0,B,C\subset_\fin M$ with $B_0\subseteq B$, if $A\dind_CB$, then $A\dind_CB_0$ and $A\dind_{CB_0}B$.

\item[] {\em Partial left-transitivity}: For $A,A_0,B,C\subset_\fin M$ with $A_0\subseteq A$, If $A_0\dind_CB$ and $A\dind_{CA_0}B$, then $A\dind_CB$.

\end{enumerate}
\end{prop}

Both Invariance and Existence are self-evident from the definitions, and Preservation-of-algebraic-dependence for $\dind$ is an easy consequence of  Preservation-of-algebraic-dependence for $\downarrow$. Partial left-transitivity is somewhat more involved, so we defer its proof to later in this section; the proof wants for Lemma \ref{lemma:non-dev-ext-to-algebraicity} below.
\begin{proof}[Proof of Extension]
Let $A,B,B_1,C\subset_\fin M$ with $BC\subseteq B_1$, and suppose $A\dind_CB$. Let $\aa$ be an enumeration of $A$, and let $p(\xx) = tp(\aa/BC)$. By the $\aleph_0$-categoricity of $T = Th(\M)$, let $p_0,...,p_{n-1}\in S(B_1)$ be an enumeration of the complete extensions of $p$ to $B_1$. Towards a contradiction, suppose that for each $i<n$, there is some $B_1\subseteq D_i\subset_\fin M$ such that $\Delta[p_i,C]_{D_i}$ is infinite. Without loss of generality, we may assume that  $\Delta[p_i,C]_{D}$ is infinite for each $i<n$, where $D = D_0\cup \cdots\cup D_{n-1}$. It follows that $\Delta[p,C]_D$ is infinite, which contradicts the $A\dind_CB$. Hence, we may certainly choose some $p_i$ and a realization $\aa'\models p_i$ such that $\aa'\dind_CB_1$, as desired.
\end{proof}

\begin{proof}[Proof of Partial right-transitivity]
Let $A,B,B_0,C\subset_\fin M$ with $B_0\subseteq B$, and suppose $A\dind_CB$. We show that $A\dind_{CB_0}B$. Let $E\in \acl[\M]_D\setminus \Delta_0[\tp(\aa/BC),C]_D$, where $BC\subseteq D\subseteq E$ and $\aa$ enumerates $A$. We obtain an $\aa_1\in E^{|\aa|}$ such that $\aa_1\models \tp(\aa/BC)$, $\aa_1\downarrow_{BC}D/E$ and $\aa_1\downarrow_CD/E$. BY $\downarrow$-base-monotonicity, we have $\aa_1\downarrow_{CB_0}D/E$, so $\aa_1$ violates the condition, ``$\aa_1\downarrow_{CB_0}D/E\implies \aa_1\nlocsep_CD/E$.'' We have shown, then, that 
$$\acl[\M]_D\setminus \Delta_0[\tp(\aa/BC),C]_D\subseteq \acl[\M]_D\setminus \Delta_0[\tp(\aa/BC),CB_0]_D$$
and with $A\dind_CB$, it follows that $A\dind_{CB_0}B$, as desired. The proof that $A\dind_CB_0$ is similar, so we omit that portion.
\end{proof}

%

%

\begin{lemma}\label{lemma:bad-seq-maker}
Let $A,B,C\subset_\fin M$ where $|A|, |B|\leq r$, and suppose $A\dind_CB$ but $B\ndind_CA$. Then $B\nsubseteq \acl(AC)$.
\end{lemma}
\begin{proof}
Towards a contradiction, suppose $B\subseteq \acl(AC)$. Let $\aa$ and $\bb$ be enumerations of $A$ and $B$, respectively. From the hypothesis $B\ndind_CA$, we may take $\Delta[p(\aa,\yy),C]_{D\aa}$ to be infinite, where $p = tp(\aa\bb/C)$ and $C\subseteq D\subset_\fin M\setminus A$. From preservation of algebraic dependence, we know that $B\ndind_{AC}AD$. Applying Extension and Invariance repeatedly, we implement the following construction:

\noindent\emph{Construction}:
\begin{itemize}
\item[] Stage 0: Set $\aa_0 = \aa$
\item[] Stage $i{+}1$: At the beginning of stage $i+1$, we have $\aa_0,...,\aa_i$, pairwise distinct, such that for each $j\leq i$,
\begin{itemize}
\item $\aa_j\equiv_{BD}\aa_0$;
\item $\aa_j\dind_{BD}\,D\bb\aa_0...\aa_{j-1}$.
\end{itemize}

We then choose $\aa_{i+1}\equiv_{D\bb\aa_0...\aa_{i-1}}\aa_i$ such that $\aa_{i+1}\dind_{BD}\,D\bb\aa_0...\aa_i$.
\end{itemize}
Since $\Delta[p(\xx,\bb),C]_E$ is finite for almost every $E$ containing $C\bb$ (up to automorphisms over $C\bb$), and applying Ramsey's theorem and the $\aleph_0$-categoricity of $T$, we may also choose a chain $E_0\subseteq E_1\subseteq\cdots\subseteq E_i\subseteq\cdots \subset_\fin M$ of algebraically closed sets such that 
$$\left[ (\star,\aa_i)\amalg (\star,\aa_j)\,|\,\hdesc(\acl(D\bb);E_k\right]_{CG[E_k]}$$
 and 
 $$\left[(\star,\aa_i)\namalg\,(\star,\bb)\,|\,\hdesc(\acl(D);E_k) \right]_{CG[E_k]}$$
  whenever $i<j<k<\omega$. This contradicts Lemma \ref{lemma:bdd-local-wt}, so this lemma is proven.
\end{proof}

Lemma \ref{lemma:non-dev-ext-to-algebraicity} now follows from Lemma \ref{lemma:bad-seq-maker} and Observation \ref{obs:alg-extensions}.

\begin{lemma}\label{lemma:non-dev-ext-to-algebraicity}
Let $C\subset_\fin M$ and $\aa,\bb\in M^r$, and let $p(\xx,\yy) = tp(\aa\bb/C)$. Let $C\bb\subseteq D\subset_\fin M$, and suppose $\Delta[p(\xx,\bb),C]_D$ is infinite. If $q(\xx,\bb)$ is a complete extension of $p(\xx,\bb)$ to $D$ which does not deviate over $C\bb$, then $q(\xx,\bb)$ is algebraic.
\end{lemma}

\begin{proof}[Proof of Partial left-transitivity]
Assume $\aa_1\dind_CB$ and $\aa_2\dind_{C\aa_1}B$, and towards a contradiction, suppose $\aa_1\aa_2\ndind_CB$, where $\aa_1,\aa_2\in M^{<\omega}$ and $BC\subset_\fin M$.  By a straightforward argument using the Extension, we may work under the assumption that $\Delta[tp(\aa_1\aa_2/BC),C]_{BC}$ is infinite. By Existence and the preceding lemma, we know that $tp(\aa_1\aa_2/BC)$ must be algebraic -- i.e. $\aa_1\aa_2\in \acl(BC)$. For the contradiction, then, we will derive that $tp(\aa_1\aa_2/BC)$ is non-algebraic. 

For a (local) contradiction, suppose $tp(\aa_1\aa_2/BC)$ is algebraic. From preservation of algebraic dependence, one can show that $\aa_1\aa_2\ndind_CB$ implies $\aa_1\aa_2\notin \acl(C)$, and it follows that  $tp(\aa_1/C)$ and $tp(\aa_2/C\aa_1)$ are not {\em both} algebraic. Similarly, as $tp(\aa_1\aa_2/BC)$ is algebraic, it cannot be that both $tp(\aa_1/BC)$ and $tp(\aa_2/BC\aa_1)$ are non-algebraic. If $tp(\aa_1/BC)$ is algebraic, then we have $\aa_1\ndind_CB$, and if $tp(\aa_2/BC\aa_1)$ is algebraic, then $\aa_2\ndind_{C\aa_1}B$ -- in either case, a contradiction. Thus, $tp(\aa_1\aa_2/BC)$ must be non-algebraic -- the higher level contradiction -- which completes the proof.
\end{proof}

\begin{thm}
$\dind$ is symmetric and fully transitive. With Proposition \ref{prop:properties-without-bdd-loc-wt}, $\dind$ is an independence relation for $K$, so $K$ is rosy.
\end{thm}
\begin{proof}
By partial right- and left-transitivity, we need only demonstrate symmetry. Let $C\subset_\fin M$ and $\aa,\bb\in M^{<\omega}$, and assume $\aa\dind_C\bb$. Towards a contradiction, we suppose that $\bb\ndind_C\aa$ -- that is, let $C\subseteq D\subset_\fin M\setminus rng(\aa)$ such that $\Delta[tp(\bb/C\aa),C]_{D\aa}$ is infinite. By lemma \ref{lemma:bad-seq-maker}, we know that $p(\aa,\yy) = tp(\bb/C\aa)$ is non-algebraic, so $p(\aa,\yy)$ has a non-algebraic extension $q(\aa,\yy)$ to $D\aa$, which by lemma \ref{lemma:non-dev-ext-to-algebraicity}, must deviate over $C\aa$. Thus, we can choose $D\subseteq D'\subset_\fin M\setminus rng(\aa)$ such that $\Delta[q(\aa,\yy),C\aa]_{D\aa}$ is infinite. Using the construction from the proof of lemma  \ref{lemma:bad-seq-maker}, we obtain a sequence $(\bb_i,E_i)_{i<\omega}$ such that $(\bb_i)_{i<\omega}$ is $\acl(D'\aa)$-indiscernible, $\bb_0 \equiv_{D'\aa}\bb$ and for every $i<\omega$, 
$\bb_i\dind_{D'\aa}D'\aa\bb_0...\bb_{i-1}$ and 
$$\bb_i\downarrow_{D\aa}D'\aa/E_j\,\,\implies \bb_i\nlocsep_{C\aa}D'\aa/E_j$$
whenever $i<j<\omega$. By Ramsey's theorem, again, we may then assume that either $\bb_i\nlocsep_{D\aa}D'\aa/E_j$ for all $i<j<\omega$, of $\bb_i\nlocsep_{C\aa}D'\aa/E_j$ for all $i<j<\omega$. In either case, we derive a contradiction to lemma  \ref{lemma:bdd-local-wt}.
\end{proof}

\bibliographystyle{plain}
\bibliography{myref}

\begin{thebibliography}{10}

\bibitem{abiteboul-vianu-1995}
Serge Abiteboul and Victor Vianu.
\newblock Computing with first-order logic.
\newblock {\em Journal of Computer and System Sciences}, 50(2):309--335, 1995.

\bibitem{adler-2007}
Hans Adler.
\newblock A geometric introduction to forking and thorn-forking.
\newblock {\em Journal of Mathematical Logic}, to appear.

\bibitem{baldwin-lessmann-2002}
John~T. Baldwin and Olivier Lessmann.
\newblock Amalgamation properties and finite models in ${L}^n$-theories.
\newblock {\em Archive for Mathematical Logic}, 41(2):155 -- 167, 2002.

\bibitem{bishop-2006}
Christopher~M. Bishop.
\newblock {\em Pattern Recognition and Machine Learning}.
\newblock Springer-Verlag, 2006.

\bibitem{koponen-2001}
Marko Djordjevic.
\newblock Finite variable logic, stability and finite models.
\newblock {\em Journal of Symbolic Logic}, 66(2):837--858, 2001.

\bibitem{koponen-2006}
Marko Djordjevic.
\newblock Finite satisfiability and $\aleph_0$-categorical structures with
  trivial dependence.
\newblock {\em Journal of Symbolic Logic}, 71(3):810--830, 2006.

\bibitem{ealy-onshuus-2007}
Clifton Ealy and Alf Onshuus.
\newblock Characterizing rosy theories.
\newblock {\em Journal of Symbolic Logic}, 72(4):919--940, 2007.

\bibitem{Grohe97largefinite}
Martin Grohe.
\newblock Large finite structures with few {$L^k$}-types.
\newblock {\em Information and Computation}, 179(2):250--278, 1997.
\newblock Special issue: LICS'97.

\bibitem{me-2011b}
Cameron~Donnay Hill.
\newblock The geometry of $\emph{L}^k$-canonization \textrm{I}: Rosiness from
  efficient constructibility.
\newblock Submitted to: \emph{Annals of Pure and Applied Logic}.

\bibitem{me-2011a}
Cameron~Donnay Hill.
\newblock Super/rosy $\emph{L}^k$-theories and classes of finite structures.
\newblock Submitted to: \emph{Annals of Pure and Applied Logic}.

\bibitem{me-2012-canonization}
Cameron~Donnay Hill.
\newblock The geometry of $\emph{L}^k$-canonization \textrm{II}:
  Coordinatization and efficient model building.
\newblock {\em (Forthcoming)}, 2012.

\bibitem{hyttinen-2000}
Tappani Hyttinen.
\newblock On stability in finite models.
\newblock {\em Archive for Mathematical Logic}, 39:89--102, 2000.

\bibitem{libkin-FMTbook-2004}
Leonid Libkin.
\newblock {\em Elements of Finite Model Theory}.
\newblock Texts in Theoretical Computer Science. Springer-Verlag, 2004.

\bibitem{otto-1997}
Martin Otto.
\newblock {\em Bounded Variable Logics and Counting: A Study in Finite Models}.
\newblock Lecture Notes in Logic. Springer-Verlag, 1997.

\bibitem{otto-1997b}
Martin Otto.
\newblock Canonization for two variables and puzzles on the square.
\newblock {\em Annals of Pure and Applied Logic}, 85(3):243--282, May 1997.

\bibitem{pearl-verma-1987}
Judea Pearl and Thomas Verma.
\newblock The logic of representing dependencies by directed graphs.
\newblock Technical Report CSD 870004, R-79-II, University of California, Los
  Angeles, Cognitive Systems Laboratory, 1987.

\bibitem{poizat-1982}
Bruno Poizat.
\newblock Deux ou trois choses que je sais de ${L}_n$.
\newblock {\em The Journal of Symbolic Logic}, 47:641--658, 1982.

\bibitem{graph-isom-survey}
Ronald~C. Read and Derek~G. Corneil.
\newblock The graph isomorphism disease.
\newblock {\em Journal of Graph Theory 1}, 4:339 -- 363, 1977.

\end{thebibliography}

\bigskip

\appendix
\section{The complete invariant for $L^k$ and game tableaux}\label{app1}

In this subsection, we introduce both the complete invariant $I^k$ for $k$-variable theories of finite $\rho$-structures and the notion of a game tableau for the $L^k$-theory of a fixed finite $\rho$-structure. The latter is key in much of our analysis of computational problems around finite-variable logic. The invariant is not itself terribly useful in this work, but as well as providing the starting point for game tableaux, this is the finite-encoding of $L^k$-theories mentioned in the introduction. The material on the complete invariant can be found in \cite{libkin-FMTbook-2004} or \cite{otto-1997}, and the material on game-tableaux can be found (with non-trivial differences) in \cite{otto-1997}, where it is deployed only in relation to the 2-variable logic.

Suppose $\M$ is a finite $\rho$-structure. Then, the quotient structure $M^k/{\equiv^k}$ is essentially synonymous with the set $S^k_k(T)$ of $k$-variable $k$-types of $T$, where $T = Th^k(\M)$ is the complete $k$-variable theory of $\M$. Moreover, if the set of quantifer-free $k$-types of $\rho$ is endowed ({\em a priori} but arbitrarily) with a linear order, then there is an $\equiv^k$-invariant algorithm computing a mapping
$$\fin[\rho]\longrightarrow\fin[\{<\}]:\N\mapsto (N^k/{\equiv^k},<^\N)$$
where $(N^k/{\equiv^k},<^\N)$ is a linear order;\footnote{\,$\equiv^k$-Invariance, here, means that if $\M\equiv^k\N$, then (i) $(M^k/{\equiv^k},<^\M)$ and $(N^k/{\equiv^k},<^\N)$ have the same (finite) length, and (ii) if $\aa\in M^k$ and $\bb\in N^k$ are such that $tp^k(\aa;\M) = tp^k(\bb;\N)$, then $\aa/{\equiv^k}$ and $\bb/{\equiv^k}$ have the same position in the orderings.} this algorithm has running-time  $|N|^{O(k)}$ (see \cite{otto-1997} or \cite{libkin-FMTbook-2004}). Now, for each complete quantifier-free $k$-type $\theta(x_1,...,x_k)$ of $\rho$, let $V_\theta^{(1)}$ be a new unary predicate symbol; for each permutation $\sigma\in\emph{Sym}\,[k]$, let $P_\sigma^{(2)}$ be a new binary predicate symbol; and let $\emph{Acc}^{(2)}$ be an additional unary predicate symbol. Let $\rho^\inv$ be the signature consisting of these symbols together with the binary relation symbol $<$. 
Given a finite $\rho$-structure $\M$, then, we define $I^k(\M)$ to be a $\rho^\inv$-structure with universe $M^k/\equiv^k$ as follows:

\begin{itemize}
\item $<^{I^k(\M)}$ is the linear order, $<^\M$, of $M^k/{\equiv^k}$ described above.
\item $V_\theta^{{I^k}(\M)} = \left\{\aa/{\equiv^k}:\aa\in M^k,\,\M\models\theta(\aa)\right\}$ for each quantifier-free $k$-type $\theta$.
\item For $\sigma\in\emph{Sym}\,[k]$, we put $(\aa/{\equiv^k},\bb/{\equiv^k})\in P_\sigma^{{I^k}(\M)}$ just in case,
$$(\M,(a_{\sigma(1)},...,a_{\sigma(k)}))\equiv^k(\M,\bb).$$
\item For $\aa,\bb\in M^k$, we put
$(\aa/{\equiv^k},\bb/{\equiv^k})\in \emph{Acc}^{{I^k}(\M)}$ just in case there is an element $m\in M$, such that 
$$(\M,(m,a_2,...,a_k))\equiv^k(\M,\bb).$$
Equivalently, $(\aa/{\equiv^k},\bb/{\equiv^k})\in \emph{Acc}^{{I^k}(\M)}$ if for every $(a'_1,...,a'_k)\in\aa/{\equiv^k}$, there is an $m'\in M$ such that 
$$(\M,(m',a'_2,...,a'_k))\equiv^k(\M,\bb).$$
\end{itemize}
This operator $I^k(-)$ is known in the literature as the \emph{complete invariant} for $k$-variable logic.

\begin{thm}(\cite{otto-1997})\label{thm:complete-invariant}
Let $\M$ and $\N$ be finite $\rho$-structures. Then $\M\equiv^k\N$ if and only if $I^k(\M)\cong I^k(\N)$. Moreover, the mapping $\M\mapsto I^k(\M)$ is computable by a relational Turing machine with running-time $|M|^{O(k)}$.
\end{thm}

We note that since $I^k(\M)$ is linearly ordered, there is a canonically isomorphic $\rho^\inv$-structure with universe $[n]:=\{1,...,n\}$, where $n=|I^k(\M)|<\omega$, in which the linear order is the standard one, and this transformation is computable in polynomial-time. Thus, it is not terribly abusive to write $I^k(\M)=I^k(\N)$ instead of $I^k(\M)\cong I^k(\N)$, with the understanding that we pass to this canonical/standard model.

\subsection{Game tableaux and the amalgamation theorem}

The complete invariant carries much unnecessary information, which has the tendency to obscure what is really essential for analyzing the Strong $L^k$-Canonization Problem. Thus, we reduce the complete invariant to a less informative structure called a {\em game-tableau theory}. As we have noted before, the quotient set $M^k/{\equiv^k}$ is essentially synonymous with the set $S^k_k(T)$ of $k$-variable $k$-types of the theory $T = Th^k(\M)$. 
Moreover, the \emph{accessibility relation} between $k$-variable $k$-types of $\M$ is an invariant of $T$. That is, we may consider $\emph{Acc}\subseteq S^k_k(T)\times S^k_k(T)$ such that for any $\rho$-structure $\N$, if $\N\models T$, $\bb\in N^k$, $p = tp^k(\bb;\N)$ and $(p,q)\in \emph{Acc}$, then there is a $b'\in N$, such that $tp^k(b',b_2,...,b_k;\N)=q$, and moreover, if $b''\in N$ and $q'=tp^k(b'',b_2,...,b_k;\N)$, then $(p,q')\in \emph{Acc}$.
Similarly, if $\sigma\in\emph{Sym}\,[k]$ and $tp^k(\aa;\M)=tp^k(\bb;\N)$, then 
$$tp^k(a_{\sigma(1)},...,a_{\sigma(k)};\M)=tp^k(b_{\sigma(1)},...,b_{\sigma(k)};\N).$$
In fact, as we shall shortly see, these facts together with some types in the language of equality effectively determine the class of models of the theory $T$.

Let $\M_0$ be a fixed finite $\rho$-structure, and let $T = Th^k(\M_0)$. We enumerate $S^k_k(T)$ by $\alpha_1(\xx),...,\alpha_N(\xx)$ ordered according to $I^k(\M_0)$, where $\xx = (x_1,...,x_k)$ is a tuple of pairwise distinct variables. For each $\alpha\in S^k_k(T)$, let $R_\alpha^{(k)}$ be a $k$-ary relation symbol, and let $\mu_\alpha(\xx)$ be the unique complete (quantifier-free) $k$-type in the language of equality such that $T\models\forall \xx(\alpha(\xx)\cond \mu_\alpha(\xx))$. Let $\rho^G=\left\{R_\alpha:\alpha\in S^k_k(T)\right\}$. We define $T^G$ to be the theory in the language of $\rho^G$ consisting of the following assertions (which obviously comprise an $\forall\exists$-theory):
\begin{enumerate}
\item[] G1: 
$\forall x_1...x_k\bigvee_\alpha\left(R_\alpha(\xx)\wedge\neg\bigvee_{\beta\neq\alpha}R_\beta(\xx)\right)$

\item[] G2: The ``type'' $R_\alpha$ of a $k$-tuple matches the equality type of the genuine type $\alpha$:
 $$\bigwedge_\alpha\forall x_1...x_k\left(R_\alpha(\xx)\cond\mu_\alpha(\xx)\right).$$
 
\item[] G3: $\bigwedge_{\sigma\in \emph{Sym}\,[k]}\bigwedge_\alpha\forall x_1...x_k\left(R_\alpha(\xx)\bic R_{\alpha^\sigma}(x_{\sigma(1)},...,x_{\sigma(k)})\right)$.

We write $\alpha^\sigma$ for the unique type $\beta$ such that 
$$T\models \forall x_1...x_k\left(\alpha(\xx)\bic \beta(x_{\sigma(1)},...,x_{\sigma(k)})\right).$$

\item[] G4: $\forall x_1...x_k\forall y\left( R_\alpha(\xx)\cond \bigvee_{\beta\in \emph{Acc}(\alpha,-)}R_\beta(y,x_2,...,x_k)\right)$

\item[] G5: $\bigwedge_{\alpha}\exists x_1...x_k\left(R_\alpha(\xx)\right)$

\item[] G6: $\bigwedge_\alpha\bigwedge_{\beta\in \emph{Acc}(\alpha,-)} \forall x_1...x_k\left(R_\alpha(\xx)\cond \exists y\left(R_\beta(y,x_2,...x_k)\right)\right)$
\end{enumerate}

There is a pair of transformations, computable in relational polynomial-time,  
$$-^G:\fin[T]\longrightarrow \fin[T^G],\,-^{\mod}:\fin[T^G]\longrightarrow\fin[T]$$
which completely characterize the relationship between $T$ and $T^G$.\footnote{Here and after, $\fin[T]$ and $\fin[T^G]$ denote the classes of finite models of $T$ and $T^G$, respectively.} Firstly, suppose $\M\in \fin[T]$; we define $\M^G$ to be the $\rho^G$-structure with universe $M$ and the obvious interpretations,
$$R_\alpha^{\M^G} = \left\{\aa\in M^k:tp^k(\aa;\M)=\alpha\right\}$$
 for each $\alpha\in S^k_k(T)$. The fact that $\M^G\models T^G$ is an easy consequence of Theorem \ref{thm:complete-invariant}. Secondly, suppose $\mathfrak{A}\in\fin[T^G]$ with universe $A$. For $R^{(r)}\in\rho$, we set 
 $$R^{\mathfrak{A}^\mod}=\left\{(a_{i_1},...,a_{i_r}):(a_1,...,a_k)\in R_\alpha^{\mathfrak{A}}, T\models\forall\xx\left(\alpha(\xx)\cond R(x_{i_1},...x_{i_r})\right)\right\}.$$
It is essentially trivial to show that the $\rho$-structure $\mathfrak{A}^\mod$ is well-defined and, indeed, a model of $T$. Collecting these facts, we have:

\begin{obs}
The transformations $-^G$ and $-^\mod$ are inverses of each other; that is to say, for any $\M\in\fin[T]$ and any $\mathfrak{A}\in\fin[T^G]$,
$(\M^G)^\mod =\M$ and $(\mathfrak{A}^\mod)^G=\mathfrak{A}$.
\end{obs}

\noindent A model $\mathfrak{A}$ of $T^G$ is called a \emph{game tableau for $T$}, and the model $\mathfrak{A}^\mod$ is sometimes called the \emph{realization} of $\mathfrak{A}$. Moreover, for a given model $\M$ of $T$, the structure $\M^G$ is called the \emph{game tableau of $\M$}; thus, a model $\M$ of $T$ is the unique realization of its own game tableau. The theory $T^G$ is the \emph{theory of game tableaux of $T$}. Abusing notation slightly, we will write $T^G_\forall$ for the sub-theory consisting of the axioms G1 through G4.

Since there is nothing interesting to distinguish a model of $T$ from its game tableau and the transformation is polynomial-time computable, it is not really necessary to distinguish between finite models of $T$ and finite models $T^G$; consequently, we will also dispense with the gothic script. In the next section, we will see that working with game tableaux makes a model-theoretic analysis much more tractable than would be the case in the original signature. The correspondence goes just a bit further in the following proposition (whose proof we omit because it is very simple).

\begin{prop}\label{T^G-preservation}
Let $\M$ and $\N$ be models of $T$, and let $A\subseteq M$. For any mapping $f:A\to N$, the following are equivalent:
\begin{enumerate}
\item $f$ is a partial $L^k$-elementary embedding $\M\pto\N$.
\item $f$ is a partial $\rho^G$-isomorphism $\M^G\pto \N^G$.
\end{enumerate}
\end{prop}

In particular, if $\M$ is a model of $T$, then the complete {\em quantifier-free} type $qtp(\aa;\M^G)$ of a tuple $\aa$ in the sense of $\M^G$ is equivalent, for our purposes, to the complete $k$-variable type $tp^k(\aa;\M)$.

\subsection*{Capped theories and amalgamation}

We will say that $T$ is a \emph{capped theory} if for any finite model $\A$ of $T^G_\forall$, there is a finite model $\mathcal{G}$ of $T^G$ such that $\A\leq\mathcal{G}$ -- that is, such that $\A$ is an induced substructure of $\mathcal{G}$. By proposition \ref{T^G-preservation}, any $T$ for which the Strong $L^k$-Canonization problem is solvable must be capped.

\begin{lemma}\label{lemma:weak-G-amalgamation}
Suppose $\A$, $\M_0$ and $\M_1$ are models of $T^G$. Suppose $\A$ is a substructure of both $\M_0$ and $\M_1$, and $M_0\cap M_1=A$. Then there is a model $\C$ of $T^G_\forall$ and $\rho^G$-embeddings $g_i:\M_i\to\C$ such that $g_0\r A = g_1\r A$.
\end{lemma}

\begin{proof}
The idea of the proof is to construct a sort of free-join of $\M_0$ and $\M_1$ over $A$. It will not be a genuine free-join because even $T^G_\forall$ may induce some additional equalities of elements, and the modified equality relation will then be a non-trivial equivalence relation, say $E$, on $M_0\cup M_1$. It's key, then, to maintain the condition $E\cap (M_i\times M_i)=1_{M_i}$, $i<2$, in order to avoid obstructing the embeddings. It turns out that maintaining this invariant through the construction is actually sufficient to obtain the amalgam over $A$.

Let $Z=M_0\cup M_1$,  and let $Q^0_\alpha = R_\alpha^{\M_0}\cup R_\alpha^{\M_1}$ for each $\alpha\in S^k_k(T)$. Furthermore, set $X_0 = Z^k\setminus \bigcup_\alpha Q^0_\alpha$ and $E_0=1_{M_0}\cup 1_{M_1}$. Suppose we are then given,
$$X_s\subsetneqq X_{s-1}\subsetneqq\cdots\subsetneqq X_0$$
such that if $\cc\in X_s$ and $\sigma\in \emph{Sym}\,[k]$, then $(c_{\sigma(1)},...,c_{\sigma(k)})\in X_s$, and 
$$E_s\supsetneqq E_{s-1}\supsetneqq\cdots\supsetneqq E_0$$
where $E_s$ is an equivalence relation on $Z$ such that $E_s\cap (M_i\times M_i)=1_{M_i}$ for $i=0,1$. Let $0<t<k$, and let $c_1,...,c_t\in M_0$ and $c_{t+1},...,c_k\in M_1$ such that $\cc\in X_s$. Let $\eta_0,\eta_1\in S^k_k(T)$ such that
$$\M_0^\mod\models \eta_0(c_1,...,c_t,c_t,...,c_t)$$
and 
$$\M_1^\mod\models \eta_1(c_{t+1},...,c_k,c_k,...,c_k).$$
For brevity, we identify $\eta_0(\xx)$, which asserts $\bigwedge_{i=t+1}^kx_i=x_t$, with the $t$-type it asserts on $x_1,...,x_t$, and similarly for $\eta_1$.
We then take the following actions:

\begin{enumerate}

\item Let $\alpha\in S^k_k(T)$ such that 
$$T\models\forall \xx\left(\alpha(\xx)\cond \eta_0(x_1,...,x_t)\wedge \eta_1(x_{t+1},...,x_k)\right).$$
Set
$$Q^{s+1}_{\alpha} = Q^s_{\alpha}\cup (\eta_0(M_0^t)\times\eta_1(M_1^{k-t}))$$
defining $Q^{s+1}_{\alpha^\sigma}$ analogously for each $\sigma\in\emph{Sym}\,[k]$.

\item Let $E_{s+1}$ be the $\subseteq$-minimal equivalence relation on $Z$ containing $E_s$ and each $(c_i,c_j)$, $i\leq t<j$, such that 
$T\models\forall \xx\left(\alpha(\xx)\cond x_i=x_j\right)$.

\end{enumerate}

\begin{claim}
We can choose $\alpha$ so that $E_{s+1}\cap (M_i\times M_i)=1_{M_i}$, $i=0,1$.
\end{claim}

\begin{proof}[proof of claim] 
We prove the claim for $i=0$; the other statement follows by symmetry. Note that we may assume $s>0$.
Suppose $a,b\in M_0$, $a\neq b$ and $aE_{s+1}b$. We may assume that $(a,b)\in E_{s+1}\setminus E_s$ and that $aE_sc$ and $bE_{s+1}c$ for some $c\in M_1$. In particular, there are (w.l.o.g.) elements
\begin{align*}
a_1=a,a_2,...,a_{t'} &\in M_0\\
c'_{t'+1}=c,c'_2,...,c_{k-t'} &\in M_1\\
b_1=b,b_2,...,b_t&\in M_0\\
c_{t+1}=c,c_2,...,c_{k-t} &\in M_1
\end{align*}
such that at step $s-1$, we acted on 
$$\zeta_0=tp^k(pad_k(\aa);\M_0^\mod),\,\zeta_1=tp^k(pad_k(\cc');\M_1^\mod)$$
 and at step $s$ (as above), we acted on 
 $$\eta_0=tp^k(pad_k(\bb);\M_0^\mod),\,\eta_1=tp^k(pad_k(\cc);\M_1^\mod).$$
 Since $\zeta_0\wedge\zeta_1\models x_1=x_{t'+1}$ and $\eta_0\wedge\eta_1\models x_1=x_{t+1}$, we now that 
 $$tp^k(a;\M_0^\mod)=tp^k(c;\M_1^\mod)=tp^k(b;\M_0^\mod).$$
 As $\M_0$ is a model of $T^G$, there are 
 $a'_1=a,a'_2,...a'_t\in M_0$ such that 
 $$tp^k(pad_k(\aa');\M_1^\mod)=\eta_0.$$
 Again, because $\M_0$ is a model of $T^G$, there are $d_{t+1},...,d_k\in M_0$ such that $tp^k(pad_k(\dd);\M_0^\mod)$ is equal to $\eta_1$.
 Now, 
\begin{align*}
\eta_0(\aa')\wedge\eta_1(\dd) \,\,&\implies\,\, a=d_{t+1}\\
\eta_0(\bb)\wedge\eta_1(\dd) \,\,&\implies\,\, b=d_{t+1}
\end{align*}
so in fact, $a=b$, a contradiction.
\end{proof}

\noindent
Since $Z$ is finite, there is a number $n<\omega$ such that $X_n=\emptyset$. (In fact, $n\leq |S^k_k(T)|^2$.) Let $C = Z/E_n$, and for $\alpha\in S^k_k(T)$, let 
$$R_\alpha^\C=\left\{(c_1/E_n,...,c_k/E_n) :(c_1,...,c_k)\in Q^n_{\alpha}\right\}.$$
For $i=0,1$, define $g_i:B_i\to C$ by $g_i(b)=b/E_n$. 
It remains to verify that the triple $(\C,g_0,g_1)$ satisfies the requirements of the lemma. 
 G1: For each $k$-tuple $\cc=(c_1,...,c_k)\in Z^k$, either $\cc\in M_0^k\cup M_1^k$ or $\cc\in X_{s-1}\setminus X_s$ for some unique $s\leq n$; hence, $\cc$ is certainly assigned a unique type. 
G2 is immediate from the claim we proved above, and G3 follows directly from the construction.
 G4 is just plain old immediate.
Finally, it's relatively easy to see that $g_0$ and $g_1$ are $\rho^G$-embeddings that agree on $A$ (in fact, each is the identity map on $A$).
\end{proof}

The lemma, together with the assumption that $T$ is a capped theory, easily yields the following very useful fact.

\begin{thm}[Amalgamation theorem: AP/models in {$\fin [T^G]$}]
Assume that $T$ is capped. Suppose $\A$, $\M_0$ and $\M_1$ are models of $T^G$. Suppose $\A$ is a substructure of both $\M_0$ and $\M_1$, and $M_0\cap M_1=A$. 
 Then there are a model $\N$ of $T^G$ and $\rho^G$-embeddings $g_i:\M_i\to\N$ such that $g_0\r A=g_1\r A$.
\end{thm}

\section{Directed acyclic graphs and d-separation.}\label{app:dags}

\begin{defn}[Dags and descendants]
We recall a few standard definitions around graphs.
\begin{enumerate}
\item A digraph is a pair $G = (V,E)$, where $V$ is a nonempty set and $E\subseteq V\times V$ is such that $(v,v)\notin E$ whenever $v\in V$.

The underlying (undirected) graph of $G$ is, then, $\tilde G = (V,\tilde E)$ where $\tilde E = \left\{\{u,v\}:(u,v)\in E\right\}$.

Where convenient, we will understand digraphs as structures with signature $\{R^{(2)}\}$; so, if $G = (V,E)$ is a digraph, the associate $\{R\}$-structure is $G = (V,R^G)$, where $R^G = E$.

\item A \emph{path} in $G$ is a sequence $(v_1,...,v_n)$ from $V^{<\omega}$ such that  $(v_i,v_{i+1})\in E$ for each $i=1,...,n-1$.

Similarly, {path} in $\tilde G$ is a sequence $(v_1,...,v_n)$ from $V^{<\omega}$ such that  $\{v_i,v_{i+1}\}\in \tilde E$ for each $i=1,...,n-1$.

\item The digraph $G$ is \emph{acyclic} if there is no path $(v_1,...,v_n)$ in $G$ such that $v_1 =v_n$. (Note that this does not imply that $\tilde G$ is acyclic.)

An acyclic digraph $G$ is also called a \emph{dag} (for ``directed acyclic graph''). 

\item In a dag $G$, we say that $v\in V$ is a {\em (proper) descendant} of $u\in V$ just in case there is a path $(v_1,...,v_n)$ in $G$ such that $v_1 = u$ and $v_n = v$. For a subset $X\subseteq V$, we define 
$$\desc_G(X) = X\cup \bigcup_{x\in X}\left\{v\in V: \textnormal{$v$ is a descendant of $x$}\right\}.$$
It can be shown that $G = (V,E)$ is a dag if and only if the transitive closure of $E$ is a partial order, say $<_E$, of $V$; in this scenario, then, $\desc_G(X) = \left\{v\in V:\exists x\in X.\, x\leq_E v\right\}$

\item Assuming $G$ is a dag, a \emph{trail} in $G$ is a sequence $(v_1,...,v_n)$ from $V^{<\omega}$ which is a path in the $\tilde G$. That is, for each $i=1,...,n-1$, either $(v_i,v_{i+1})\in R^G$ or $(v_{i+1},v_i)\in R^G$. (Not both because $G$ is acyclic.)
\end{enumerate}
\end{defn}

\begin{defn}[d-Separation in dags]
Let $G=(G,R^G)$ be a dag, and let $\tilde G = (G,\tilde R^G)$ be the underlying undirected graph of $G$.
\begin{enumerate}

\item Let $t = (v_1,...,v_n)$ be a trail in $G$, and let $i\in \{2,...,n-1\}$.
\begin{enumerate}
\item $t$ is {\em head-to-tail} at $v_i$ if either $G\models R(v_{i-1},v_i)\wedge R(v_i,v_{i+1})$ or $G\models R(v_{i+1},v_i)\wedge R(v_i,v_{i-1})$.
\item $t$ is {\em tail-to-tail} at $v_i$ if $G\models R(v_i,v_{i-1})\wedge R(v_i,v_{i+1})$.
\item $t$ is {\em head-to-head} at $v_i$ if $G\models R(v_{i-1}v_i)\wedge R(v_{i+1},v_i)$.

\end{enumerate}

\item Let $Z\subseteq G$, and let  $t = (v_1,...,v_n)$ be a trail in $G$. We say that $t$ is {\em $Z$-blocked} if for some $i\in\{2,...,n-1\}$, one of the following holds:
\begin{enumerate}
\item $t$ is head-to-tail or tail-to-tail at $v_i$ and $v_i\in Z$;
\item $t$ is head-to-head at $v_i$ and $Z\cap \desc_G(\{v_i\}) = \emptyset$.
\end{enumerate}

\item Let $X,Y,Z\subseteq G$. We say that {\em $X$ and $Y$ are d-separated by $Z$ (in $G$)} if
\begin{enumerate}
\item $X\cap Y\subseteq Z$;
\item  For every trail  $t = (v_1,...,v_n)$ in $G$ such that $v_1\in X\setminus Z$ and $v_n\in Y\setminus Z$, $t$ is $Z$-blocked.
\end{enumerate}
  We write $[X\amalg Y\,|\,Z]_G$ to indicate that $X$ and $Y$ are d-separated by $Z$ in $G$.
\end{enumerate}
\end{defn}

The notion of d-separation seems to have arisen, originally, in statistical learning theory -- as in \cite{pearl-verma-1987}. In that domain, d-separation in a finite dag $G=(V,E)$ is a means of representing conditional independence under a joint probability distribution $p$ on a system of random variables $(X_v)_{v\in V}$. Let $\texttt{parents}_G(v) = \left\{u\in V:(u,v)\in E\right\}$, the graph structure is taken as a synonym for the assertion,
$$p((X_v=a_v)_v) = \prod_{v\in V}p(X_v=a_v\,|\,(X_u = b_u)_{u\in \texttt{parents}_G(v)}).$$
One might, then, be interested in the conditional independence properties of marginals of $p$. Let us write $p(A=a)$ as shorthand for the marginal distribution $p((X_v=a_v)_{v\in A})$, where $A\subseteq V$ and $a\in \prod_{v\in V}\Sigma_v$ so that each $v\in V$ corresponds to a $\Sigma_v$-valued random variable. Then, given $A,B,C\subseteq V$, one might ask if 
$p(A=a, B=b\,|\, C=c) = p(A=a\,|\,C=c)\cdot p(B=b\,|\,C=c)$. It can be shown, as in \cite{pearl-verma-1987} and \cite{bishop-2006}, that if $A,B,C$ are pairwise disjoint, then $p(A=a, B=b\,|\, C=c) = p(A=a\,|\,C=c)\cdot p(B=b\,|\,C=c)$, for all $a,b,c$, if and only if $[A\amalg B\,|\,C]_G$ holds.

\begin{thm}
Let $G = (G,R^G)$ be a dag.
\begin{itemize}
\item[] 0: $[X\amalg Y\,|\,Z]_G\,\,\iff\,\,X\cap Y\subseteq Z\wedge [\,(X\setminus Y)\amalg (Y\setminus Z)\,|\,Z]_G$
\item[] \emph{Symmetry}: $[X\amalg Y\,|\,Z]_G\,\,\implies\,\, [Y\amalg X\,|\,Z]_G$
\item[]  \emph{Monotonicity}: $[X\amalg Y\,|\,Z]_G\wedge Y_0\subseteq Y\,\,\implies\,\,[X\amalg Y_0\,|\,Z]_G$
\item[]  \emph{Base-monotonicity}: $[X\amalg Y\,|\,Z]_G\wedge Y_0\subseteq Y \,\,\implies\,\,[X\amalg (Y\setminus Y_0)\,|\,Z\cup Y_0]_G$
\item[]  \emph{Triviality}: $[X\amalg Y_1\,|\,Z]_G\wedge [X\amalg Y_1\,|\,Z]_G\,\,\implies \,\,[X\amalg Y_1\cup Y_2\,|\,Z]_G$
\end{itemize}
\end{thm}
The names of the properties are chosen, here, to correspond to those similar properties of model-theoretic independence relations, and consequently, they {\em are not} the names used in the statistical learning theory literature. In the latter, the names for Symmetry, Monotonicity, Base-monotonicity, and Triviality are, respectively, Symmetry, Decomposition, Weak Union, and Contraction.

%
%
%
%

\end{document}